\documentclass[12pt]{article}
\usepackage{amsmath,amssymb,color,amscd}
\usepackage{xspace}

\newcommand{\Var}{{\rm{Var}_{\mathbb{C}}}}
\newcommand{\Vect}{{\rm{Vect}_{\mathbb{C}}}}

\newcommand{\ind}{{\mbox{\rm ind}}}

\newcommand{\diag}{{\mbox{\rm diag}}}

\def\1{\underline{1}}

\def\LLL{{\mathbb L}}
\def\Z{{\mathbb Z}}

\def\Q{{\mathbb Q}}
\def\C{{\mathbb C}}

\def\S{{\mathcal S}}

\def\calX{{\mathcal X}}
\def\calY{{\mathcal Y}}

\DeclareMathOperator{\spec}{Spec}

\def\ind{{\rm ind}}

\newtheorem{theorem}{Theorem}

\newtheorem{lemma}{Lemma}
\newtheorem{proposition}{Proposition}
\newtheorem{definition}{Definition}

\newenvironment{corollary}
{\smallskip\noindent{\bf Corollary\/}.}{\smallskip\par}

\newenvironment{remark}
{\smallskip\noindent{\bf Remark\/}.}{\smallskip\par}
\newenvironment{remarks}
{\smallskip\noindent{\bf Remarks\/}.}{\smallskip\par}
\newenvironment{proof}
{\noindent{\bf Proof\/}.}{{ $\square$}\smallskip\par}

\title{Grothendieck ring of varieties with finite groups actions
\footnote{Math. Subject Class.: 14F30, 18F30, 55M35. Keywords: 
finite group actions, complex quasi-projective varieties, Grothendieck rings,
lambda-structure, power structure.}
}

\author{S.M.~Gusein-Zade \thanks{The work of the first author
(Sections~\ref{sec:Intro}, \ref{sec:Orbifold_Euler}, \ref{sec:power_K^fGr}, and~\ref{sec:Groth_vector}) 
was supported by the grant 16-11-10018 of the Russian Science Foundation.
Address: Moscow State University, Faculty
of Mechanics and Mathematics, GSP-1, Moscow, 119991, Russia. E-mail:
sabir\symbol{'100}mccme.ru} \and I.~Luengo \thanks{The last two authors were partially
supported by a competitive Spanish national grant MTM2016-76868-C2-1-P.
Address:  ICMAT (CSIC-UAM-UC3M-UCM), Dept. of Algebra, Complutense University of Madrid,
Plaza de Ciencias 3, Madrid, 28040, Spain.
E-mail: iluengo\symbol{'100}mat.ucm.es} \and
A.~Melle-Hern\'andez \thanks{Address:  Instituto de Matem\'a¡tica Interdisciplinar (IMI), Dept. of Algebra, Complutense University of Madrid,
Plaza de Ciencias 3, Madrid, 28040, Spain. E-mail: amelle\symbol{'100}mat.ucm.es}}
\date{}
\begin{document}
\def\eps{\varepsilon}

\maketitle

\begin{abstract}
We define a Grothendieck ring of varieties with finite groups actions and show that the
orbifold Euler characteristic and the Euler characteristics of higher orders can be defined
as homomorphisms from this ring to the ring of integers. We describe two natural $\lambda$-structures
on the ring and the corresponding power structures over it and show that one of these power structures
is effective. We define a Grothendieck ring of varieties with equivariant vector
bundles and show that the generalized (``motivic'') Euler characteristics of higher orders can be defined
as homomorphisms from this ring to the Grothendieck ring of varieties extended by powers of 
the class of the complex affine line.
We give an analogue of the Macdonald type formula for the generating series of the
generalized higher order Euler characteristics of wreath products.
\end{abstract}

\section{Introduction}\label{sec:Intro}
The Euler characteristic $\chi(\cdot)$ (defined as the alternating sum of cohomology groups with compact
support) is an \emph{additive} invariant of topological spaces (sufficiently nice, e.~g., quasi-projective varieties).
It can be considered as a homomorphism from the Grothendieck ring $K_0(\Var)$ of quasi-projective varieties
to the ring $\Z$ of integers. There exists a simple formula for the generating series of the Euler
characteristics of the symmetric powers $S^nX=X^n/S_n$ of a variety $X$:
\begin{equation}\label{eq:Macdonald}
 1+\chi(X)t+\chi(S^2X)t^2+\chi(S^3X)t^3+\ldots=(1-t)^{-\chi(X)}
\end{equation}
(the Macdonald formula). (One can interpret this formula as the fact that the Euler characteristic is
a $\lambda$-ring homomorphism between the rings $K_0(\Var)$ and $\Z$ endowed with natural $\lambda$-structures:
see Section~\ref{sec:Power_lambda}).

For a topological space $X$ with an action of a finite group $G$, one has the notions of the orbifold Euler
characteristic $\chi^{\rm orb}(X,G)$ (coming from physics) and of Euler characteristics $\chi^{(k)}(X,G)$
of higher orders ($\chi^{\rm orb}(X,G)=\chi^{(1)}(X,G)$). They can be considered as ($\Z$-valued) functions
on the Grothendieck ring $K_0^G(\Var)$ of $G$-varieties. However, they are not ring homomorphisms
from $K_0^G(\Var)$ to $\Z$.

In \cite{GPh}, there were defined generalized (``motivic'') versions of the orbifold Euler characteristic
and of Euler characteristics of higher orders with values in the Grothendieck ring $K_0(\Var)$ of varieties
extended by the rational powers of the class $\LLL$ of the complex affine line.
They are defined for a non-singular (!) variety with an action of a finite group $G$. They are not defined as
functions on a certain ring (say, on the Grothendieck ring of $G$-varieties). There was obtained a Macdonald
type formula for the generating series of generalized higher order Euler characteristics of the Cartesian
powers of a $G$-manifold with the actions of the wreath products $G_n$ on them. It is formulated in terms of
the so-called power structure over the ring $K_0(\Var)[\LLL^s]$. (A power structure over a ring $R$ is a
method to give sense to an expression of the form $(1+a_1t+a_2t^2+\ldots)^m$, where $a_i$ and $m$ are elements
of the ring $R$: \cite{GLM-MRL}, see also Section~\ref{sec:Power_lambda}.)

Here we define a Grothendieck ring $K^{\rm fGr}_0(\Var)$ of varieties with finite groups actions.
Elements of $K^{\rm fGr}_0(\Var)$ are classes of varieties with actions of finite groups (different, in general)
on subvarieties constituting partitions of them. The most important ingredient of the definition is
the identification of the class of a variety with a group action with the class of the variety obtained by
the induction operation (with an action of a bigger group).
Another description of this ring gives elements of it as
classes of $G$-varieties with ``a group $G$ big enough''. More precisely they are elements of the injective
limit of the Grothendieck groups (not rings!) of $G$-varieties with respect to the inclusion of finite groups.

We show that the orbifold Euler characteristic and the Euler characteristics of higher orders can be defined
as homomorphisms from $K^{\rm fGr}_0(\Var)$ to the ring $\Z$ of integers.
We describe two natural $\lambda$-structures on the ring
$K^{\rm fGr}_0(\Var)$. These $\lambda$-structures define different power structures over this ring.
We show that one of this power structures is not effective (see the definition in
Section~\ref{sec:Power_lambda}) and the other one is. We give a geometric description of the effective
power structure. We define a Grothendieck ring $K^{\rm fGr}_0(\Vect)$ of varieties with equivariant vector
bundles and show that the generalized Euler characteristics of higher orders can be defined as homomorphisms
from $K^{\rm fGr}_0(\Vect)$ to $K_0(\Var)[\LLL^s]$.
We give an analogue of the formula from \cite{GPh} for the generating series of the
generalized higher order Euler characteristics of wreath products.

\section{Orbifold Euler characteristics and their generalized versions}\label{sec:Orbifold_Euler}
For a finite group $G$, let ${\rm Cong\,}G$ be the set of the conjugacy classes of elements of $G$,
for an element $g\in G$ let $C_G(g)=\{h\in G: h^{-1}gh=g\}$ be the centralizer of $g$.
For a $G$-space $X$ and for a subgroup $H\subset G$, let $X^H=\{x\in X: gx=x \text{\ \ for all\ \ } g\in H\}$
be the fixed point set of the subgroup $H$.
The {\em orbifold Euler characteristic}
$\chi^{orb}(X,G)$ of the $G$-space $X$ is defined, e.~g., in \cite{AS}, \cite{HH}:
\begin{equation}\label{chi-orb}
\chi^{orb}(X,G)=
\frac{1}{\vert G\vert}\sum_{{(g_0,g_1)\in G\times G:}\atop{\\g_0g_1=g_1g_0}}\chi(X^{\langle g_0,g_1\rangle})
=\sum_{[g]\in {{\rm Conj\,}G}} \chi(X^{\langle g\rangle}/C_G(g))\in \Z\,,
\end{equation}
where $g$ is a representative of a class $[g]$, $\langle g\rangle$ and $\langle g_0,g_1\rangle$
are the subgroups of $G$ generated by the corresponding elements.

The {\em higher order Euler characteristics} of $(X,G)$ were defined in \cite{AS} and \cite{BrF} by:
\begin{equation}\label{chi-k-orb}
 \chi^{(k)}(X,G)=
\frac{1}{\vert G\vert}\sum_{{{\bf g}\in G^{k+1}:}\atop{g_ig_j=g_jg_i}}\chi(X^{\langle {\bf g}\rangle})
=\sum_{[g]\in {{\rm Conj\,}G}} \chi^{(k-1)}(X^{\langle g\rangle}, C_G(g))\,,
\end{equation}
where $k$ is a positive integer (the order of the Euler characteristic),
${\bf g}=(g_0,g_1, \ldots, g_k)$, $\langle{\bf g}\rangle$ is the subgroup 
generated by $g_0,g_1, \ldots, g_k$, and (for the second, recurrent, definition)
$\chi^{(0)}(X,G)$ is defined as the usual Euler characteristic $\chi(X/G)$ of the quotient.
The orbifold Euler characteristic $\chi^{orb}(X,G)$
is the Euler characteristic $\chi^{(1)}(X,G)$ of order $1$.

For a $G$-variety $X$, the cartesian power $X^n$ carries an action of the wreath product
$G_n=G^n \rtimes S_n$ generated by the natural action of the symmetric group $S_n$ (permuting
the factors) and by the natural (component-wise) action of the Cartesian power $G^n$.
The pair $(X_n, G_n)$ should be (or can be) considered as an analogue of the symmetric power
for the pair $(X,G)$.

For $k\ge 0$ one has the following Macdonald type formula (see \cite[Theorem A]{Tamanoi})
\begin{equation}\label{higher_generating}
1+\sum_{n=1}^{\infty}\chi^{(k)}(X^n, G_n)\cdot  t^n
=\left(
\prod\limits_{r_1, \ldots,r_k\geq 1}\left(1-t^{r_1r_2\cdots r_k}\right)^{r_2r_3^2\cdots r_k^{k-1}}
\right)
^{-\chi^{(k)}(X, G)}\,.
\end{equation}
When $k=0$, one gets the standard Macdonald formula (Equation~(\ref{eq:Macdonald})) for the quotient $X/G$.

There is a (more or less) natural notion of the Grothendieck ring $K_0^G(\Var)$ of (complex quasi-projective)
$G$-varieties such that the orbifold Euler characteristic $\chi^{\rm orb}(\cdot)$ and the Euler
characteristics of higher orders $\chi^{(k)}(\cdot)$ are functions on it.
{\em The Grothendieck ring $K_0^G(\Var)$ of complex quasi-projective $G$-varieties} is the Abelian group
generated by the $G$-isomorphism classes $[X,G]$ of complex quasi-projective varieties
$X$ with $G$-actions modulo the relation:
$[X,G]=[Y,G]+[X\setminus Y,G]$ for a Zariski closed $G$-invariant subvariety $Y$ of $X$.
The multiplication in $K_0^{G}(\Var)$ is defined by the Cartesian product with the diagonal $G$-action.

\begin{remark}
Usually, in the definition of the Grothendieck ring of complex quasi-projective $G$-varieties,
one adds the following relation:
if $E\to X$ is a $G$-equivariant vector bundle of rank $n$, then $[E]=\LLL^n\cdot[X,G]$,
where $\LLL$ is the class of the complex affine line with the trivial $G$-action.
We use the definition given, e.~g., in \cite{Mazur}. In \cite{bittner} this Grothendieck ring
was also defined (alongside with the ``traditional one'') and was denoted by $K_0^{{\,}', G}(\Var)$. 
\end{remark}

One can easily understand that $\chi^{(k)}$ are additive functions on $K_0^G(\Var)$, however, they are not
multiplicative. This can be seen, e.~g., from the fact that $\chi^{(k)}(1)=\vert G\vert$ ($1\in K_0^G(\Var)$
is the class of the one-point variety with the only $G$-action on it). Thus they are not ring homomorphisms
from $K_0^G(\Var)$ to $\Z$. In what follows we, in particular, define a Grothendieck ring (so-called
Grothendieck ring of varieties with finite group{\bf s} actions) such that $\chi^{\rm orb}$ and
$\chi^{(k)}$ are ring homomorphisms from it to $\Z$.

Let $K_0(\Var)[\LLL^s]_{s\in\Q}$ (or $K_0(\Var)[\LLL^s]$ for short) be the modification of
the Grothendieck ring $K_0(\Var)$ of quasi-projective varieties obtained by adding all rational
powers of the class $\LLL$ of the complex affine line. The elements
of $K_0(\Var)[\LLL^s]$ are in a bijective correspondence with the finite sums of the form
$\sum_i c_i \LLL^{r_i}$, where $c_i$ are elements of the localization $K_0(\Var)_{(\LLL)}$
of the ring $K_0(\Var)$ by the class $\LLL$, $r_i$ are different rational numbers inbetween
$0$ and $1$: $0\le r_i <1$. Thus the ring $K_0(\Var)[\LLL^s]$ contains the localization
$K_0(\Var)_{(\LLL)}$. It was shown (L.~Borisov) that $\LLL$ is a zero divisor in $K_0(\Var)$
and therefore the natural map $K_0(\Var)\to K_0(\Var)_{(\LLL)}$ is not injective.
Therefore the natural map $K_0(\Var)\to K_0(\Var)[\LLL^s]$ is not injective as well.

Let $X$ be a non-singular complex quasi-projective variety of dimension $d$ with an (algebraic) action of
the group $G$. To define the  {\em higher order generalized (``motivic'') Euler characteristics}
of the pair $(X,G)$, one has to use the so called age (or fermion shift) ${\rm age}_x(g)$ of an element
$g\in G$ at a fixed point $x$ of $g$ defined in \cite{Zaslow}, \cite{Ito-Reid}.
The element $g$ acts on  the tangent space $T_xX$ as an automorphism of finite order. This action
on $T_xX$ can be represented by a diagonal matrix
$\diag(\exp(2\pi i \theta_1), \ldots, \exp(2\pi i \theta_d))$
with $0\le\theta_j<1$ for $j=1,2, \ldots, d$ ($\theta_j$ are rational numbers). The {\em age}
of the element $g$ at the point $x$
is defined by ${\rm age}_x(g)=\sum_{j=1}^{d}\theta_j\in\Q_{\ge 0}$. 
For a rational number $q$, let $X^{\langle g\rangle}_q$ be the set of points $x\in X^{\langle g\rangle}$
such that ${\rm age}_x(g)=q$.

For a rational number $\varphi_1\in \Q$, the generalized orbifold Euler characteristic
of weight $\varphi_1$ of the pair $(X,G)$ is defined by
\begin{equation}\label{generalized_orbifold}
[X,G]_{\varphi_1}:=\sum_{[g]\in {{\rm Conj\,}G}}\sum_{q\in \Q}
[X^{\langle g\rangle}_{q}/C_G(g)]\cdot\,\LLL^{\varphi_1 q}\in 
K_0(\Var)[\LLL^{s}]\,.
\end{equation}
(See an explanation for introducing the weight $\varphi_1$ in \cite{GPh}. Generalized orbifold Euler
characteristic is meaningful for at least two values of the weight $\varphi_1$: $0$ and $1$.)
Equation~(\ref{generalized_orbifold}) is a reformulation of the definition from~\cite{Wang}
given in terms of the orbifold Hodge--Deligne polynomial.

The generalized Euler characteristics of higher orders are defined recursively by an equation
which is a sort of ``a motivic version'' of the second equality in (\ref{chi-k-orb}) taking into
account ages of elements: see~\cite{GPh} for the details or Section~\ref{sec:Groth_vector} for
a somewhat more general definition. Since all of them are defined only for smooth varieties, they
are not functions on a certain ring (say, on a Grothendieck ring of $G$-varieties).
In Section~\ref{sec:Groth_vector} we define a Grothendieck ring (the Grothendieck ring of varieties with
equivariant vector
bundles) such that (appropriately defined) generalized Euler characteristics of higher orders
are ring homomorphisms from this Grothendieck ring to $K_0(\Var)[\LLL^{s}]$.

A Macdonald type formula for the generalized Euler characteristics of higher orders (a ``motivic'' version
of~(\ref{higher_generating})) is written in terms of the power structure over the ring $K_0(\Var)[\LLL^{s}]$).

\section{$\lambda$-structures and power structures}\label{sec:Power_lambda}

A Macdonald type equation for an invariant taking values in a certain ring can be formulated in terms of
a power structure over the ring of values of the invariant. Let $R$ be a commutative ring
with unity. A power structure over the ring $R$ is a method to give sense to expressions of the form
$(A(t))^m$, where $A(t)=1+a_1t+a_2t^2+\ldots$ is a power series with the coefficients $a_i$ from $R$
and $m$ is an element of $R$.

\begin{definition}\cite{GLM-MRL}
 A {\em power structure} over the ring $R$ is a map
 $$
 \left(1+tR[[t]]\right)\times R\to 1+tR[[t]]\quad\quad
 ((A(t),m)\mapsto\left(A(t)\right)^m)
 $$
 possessing the properties of the exponential function, namely:
 \begin{enumerate}
\item[(1)]  $\left(A(t)\right)^0=1$;
\item[(2)]  $\left(A(t)\right)^1=A(t)$;
\item[(3)]  $\left(A(t)\cdot B(t)\right)^{m}=\left(A(t)\right)^{m}\cdot \left(B(t)\right)^{m}$;
\item[(4)]  $\left(A(t)\right)^{m+n}=\left(A(t)\right)^{m}\cdot \left(A(t)\right)^{n}$;
\item[(5)]  $\left(A(t)\right)^{mn}=\left(\left(A(t)\right)^{n}\right)^{m}$;
\item[(6)]  $(1+a_1t+\ldots)^m=1+ma_1t+\ldots$;
\item[(7)]  $\left(A(t^k)\right)^m =
\left(A(t)\right)^m\raisebox{-0.5ex}{$\vert$}{}_{t\mapsto t^k}$.
 \end{enumerate}
 \end{definition}

Let $\mathfrak{m}$ be the ideal $tR[[t]]$ in the ring $R[[t]]$.
 A power structure over the ring $R$ is {\em finitely determined} if, for any $k\ge 1$, the fact that
 $A(t)\in 1+{\mathfrak{m}}^k$ implies that $\left(A(t)\right)^m\in 1+{\mathfrak{m}}^k$.

The natural power structure over the ring $\Z$ of integers is defined by the standard formula
for a power of a series:
\begin{eqnarray*}
 &\ &(1+a_1t+a_2t^2+\ldots)^m=\\
 &=&1+\sum_{k=1}^{\infty}
 \left(\sum_{\{k_i\}:\sum ik_i=k}\frac{m(m-1)\cdots(m-\sum_i k_i +1)\cdot\prod_i a_i^{k_i}}
 {\prod_i k_i!}\right)\cdot t^k.
\end{eqnarray*}

A power structure over the Grothendieck ring $K_0(\Var)$ of complex quasi-projective varieties
was defined in \cite{GLM-MRL} by the formula
\begin{eqnarray}\label{geometric}
 &\ &(1+[A_1]t+[A_2]t^2+\ldots)^{[M]}=\nonumber\\
 &=&1+\sum_{k=1}^{\infty}
 \left(\sum_{\{k_i\}:\sum ik_i=k}
 \left[\left(\left(M^{\sum_i k_i}\setminus\Delta\right)\times\prod_i A_i^{k_i}\right)\left/
 {\prod_i S_{k_i}}\right.\right]\right)\cdot t^k,\label{Power}
\end{eqnarray}
where $A_i$, $i=1, 2, \ldots$, and $M$ are quasi-projective varieties ($[A_i]$ and $[M]$ are
their classes in the ring $K_0(\Var)$), $\Delta$ is the large diagonal in $M^{\sum_i k_i}$,
that is, the set of (ordered) $\left(\sum_i k_i\right)$-tuples of points of $M$ with at least
two coinciding ones, the group $S_{k_i}$ of permutations on $k_i$ elements acts by the
simultaneous permutations on the components of the corresponding factor $M^{k_i}$ in
$M^{\sum_i k_i}=\prod_i M^{k_i}$ and on the components of $A_i^{k_i}$.

Except the Grothendieck ring of complex quasi-projective varieties one can consider {\em the
Grothendieck semiring} $S_0(\Var)$. It is defined in the same way as $K_0(\Var)$ with the word {\em group}
substituted by the word {\em semigroup}. Two complex quasi-projective varieties $X$ and $Y$ represent the same
element of the semiring $S_0(\Var)$ if and only if they are piece-wise isomorphic, that is if there
exist decompositions $X =\bigsqcup_{i=1}^sX_i$ and $Y =\bigsqcup_{i=1}^sY_i$ into Zariski locally
closed subsets such that $X_i$ and $Y_i$ are isomorphic for $i = 1, \ldots , s$. There is a natural map
(a semiring homomorphism) from $S_0(\Var)$ to $K_0(\Var)$. (It is not known whether or not this
map is injective.)

A power structure over the Grothendieck ring $K_0(\Var)$ is called {\em effective} if the
fact that all the coefficients $a_i$ of the series $A(t)$ and the exponent $m$ are represented
by classes of complex quasi-projective varieties (i.~e., belong to the image of the map
$S_0(\Var)\to K_0(\Var)$) implies that all the coefficients of the series $\left( A(t)\right)^m$
are also represented by such classes. (Roughly speaking this means that the power structure can be
defined over the Grothendieck semiring $S_0(\Var)$.) The same concept can be considered for Grothendieck
rings of complex quasi-projective varieties with additional structures. The effectiveness of the
described power structure over the Grothendieck ring $K_0(\Var)$ is clear from Equation~(\ref{geometric}).

The notion of a power structure over a ring is closely related with the notion of a $\lambda$-ring
structure. A $\lambda$-ring structure (or a pre-$\lambda$-ring structure in a certain terminology,
see, e.~g., \cite{Knutson})
is an additive-to-multiplicative homomorphism $R \to 1+tR[[t]], a\mapsto \lambda_a(t)$ 
($\lambda_{a+b}(t)=\lambda_a(t) \cdot \lambda_b(t)$) such that $\lambda_a(t)=1+at +\ldots$
A $\lambda$-ring structure on a ring defines a finitely determined power structure over it
in the following way. Any series $A(t)\in 1+tR[[t]]$ can be in  a unique way
represented as $\prod_{i=1}^{\infty}\lambda_{b_i}(t^i)$, for some $b_i\in R$.
Then one defines  $\left(A(t)\right)^m:=\prod_{i=1}^{\infty}\lambda_{m b_i}(t^i).$
On the other hand, in general, there are many $\lambda$-ring structures corresponding to one
power structure over a ring. One can show that the power structure~(\ref{geometric}) is defined by
the $\lambda$-ring structure on the Grothendieck ring $K_0(\Var)$ given by the Kapranov zeta function 
$$
\zeta_{[M]}(t):=1+\sum_{k=1}^\infty [S^k M] \cdot t^k,
$$
where $S^k M=M^k/S_k$ is the $k$th symmetric power of the variety $M$.
This follows from the following equation
$$
\zeta_{[M]}(t)=(1-t) ^{-[M]}=(1+t+t^2+\ldots)^{[M]}.
$$
The power structure~(\ref{geometric}) over the Grothendieck ring $K_0(\Var)$ is also defined by
the following $\lambda$-ring structure on it. Let $B^k M=(M^k\setminus \Delta)/S_k$ be the configuration space
of $k$-point subsets of $M$ ($\Delta$ is the big diagonal in $M^k$ consisting on $k$-tuples of points of $M$
with at least two coinciding ones). The series
$$
\lambda_{[M]}(t):=1+\sum_{k=1}^\infty [B_k M] \cdot t^k
$$
gives a $\lambda$-ring structure on the Grothendieck ring $K_0(\Var)$ which defines the same
power structure~(\ref{geometric}) over it. In terms of the power structure one has
$\lambda_{[M]}(t)=(1+t)^{[M]}$.

Alongside with a $\lambda$-structure on a ring (defined by a series $\lambda_a(t)$) one has the so-called
opposite $\lambda$-structure defined by the series $\lambda'_a(t):=(\lambda(-t))^{-1}$. One can show that
the power structure over the ring $K_0(\Var)$ defined by the $\lambda$-structures opposite to
$\zeta_{[M]}(t)$ and $\lambda_{[M]}(t)$ is not effective:~\cite{Stacks}.

\section{Grothendieck ring of varieties with finite groups actions}\label{sec:Groth_finite_groups}
\begin{definition}\label{def:fin-groups}
A {\em quasi-projective variety $\calX$ with a finite groups action} is a variety represented as the disjoint
union of (locally closed) subvarieties $X_i$, $i=1, \ldots, s$, with (left) actions of finite groups $G_i$
on them. 
\end{definition}

This means that $\calX$ can be decomposed into parts with actions of (different, in general) finite groups
on them. We shall write $\calX=\bigsqcup\limits_{i=1}^s (X_i,G_i)$. A partition of $\calX$ means partitions
of its components $X_i$ as $G_i$-varieties. In particular a $G$-variety ($G$ is a finite group) is a
variety with a finite groups action. For short we will call varieties of this sort {\em varieties with pure
actions}.

\begin{definition}\label{def:fGr-isomorphic}
Two varieties with pure actions $(Z,G)$ and $(Z',G')$ are {\em isomorphic} if there exist isomorphisms
$\varphi:G\to G'$ (of finite groups) and $\psi:Z\to Z'$ (of quasi-projective varieties) such that $\psi$
is equivariant relative to $\varphi$, that is, $\psi(gx)=\varphi(g)\psi(x)$ for $x\in Z$, $g\in G$.
\end{definition}

\begin{definition}\label{def:fGr-equivalent}
Two varieties with finite groups actions $\calX$ and $\calY$ are called {\em equivalent} if there exist partitions
$\calX=\bigsqcup\limits_{i=1}^N(X_{(i)}, G_{(i)})$ and $\calY=\bigsqcup\limits_{i=1}^N(Y_{(i)}, G'_{(i)})$
such that $(X_{(i)}, G_{(i)})$ is isomorphic to $(Y_{(i)}, G'_{(i)})$ for $i=1, \ldots, N$.
\end{definition}

There exist a somewhat natural notion of the Grothendieck ring of varieties with a finite groups actions: see
below. However, it does not really correspond to our aim. Because of that here we will use the name
pre-Grothendieck ring.

\begin{definition}\label{def:pre-Groth-fGr} {\em The pre-Grothendieck ring of quasi-projective varieties with
finite groups actions} is the Abelian group $\widetilde{K}_0^{\rm fGr}(\Var)$ generated by the classes
$[\calX]$ of (quasi-projective) varieties with finite groups actions modulo the following relations:
 \begin{enumerate}
 \item[(1)] if quasi-projective varieties with finite groups actions $\calX$ and $\calY$ are equivalent, then
 $[\calX]=[\calY]$;
 \item[(2)] if $\calY$ is a Zariski closed subvariety invariant with respect to the groups actions  of $\calX$, 
then
 $[\calX]=[\calY]+[\calX\setminus \calY]$.
\end{enumerate}
The multiplication in $\widetilde{K}_0^{\rm fGr}(\Var)$ is defined by the Cartesian product of 
varieties with the natural finite groups action on it.
\end{definition}

In particular, for varieties with pure actions, one has
$$
\left[(Z_1,G_1)\right]\cdot\left[(Z_2,G_2)\right]=\left[(Z_1\times Z_2,G_1\times G_2)\right]
$$
with the natural (diagonal) action of $G_1\times G_2$. 
The unit element 
in $\widetilde{K}_0^{\rm fGr}(\Var)$ is $1=[ (\spec(\, \C\, ), (e))]$, the class of the one-point variety with the action of the group
with one element. The ring $\widetilde{K}_0^{\rm fGr}(\Var)$ (as an Abelian group) is generated by
the classes $[(Z,G)]$ of varieties with pure actions.

It is easy to see that the orbifold Euler characteristic $\chi^{\rm orb}$ and the Euler characteristics
$\chi^{(k)}$ of higher order can be defined as functions on the Grothendieck ring
$\widetilde{K}_0^{\rm fGr}(\Var)$. Moreover, the following statement implies that they are (ring)
homomorphisms from $\widetilde{K}_0^{\rm fGr}(\Var)$ to $\Z$.

\begin{proposition}\label{prop:tamanoi} (see, e.~g., \cite[Proposition 2--1]{Tamanoi})
For two varieties with pure actions $(Z,G)$ and $(Z',G')$ one has
$$
\chi^{(k)}(Z\times Z',G\times G')=\chi^{(k)}(Z,G)\cdot\chi^{(k)}(Z',G')\,.
$$
\end{proposition}

Nevertheless it is not clear whether the ring $\widetilde{K}_0^{\rm fGr}(\Var)$ can be endowed with a
(natural) $\lambda$-structure and therefore it is not possible to try to consider $\chi^{\rm orb}$ and
$\chi^{(k)}$ as $\lambda$-ring homomorphisms. To make this possible we have to introduce a sort of
a reduction of the ring $\widetilde{K}_0^{\rm fGr}(\Var)$.

Let $Z$ be a $G$-variety and let $H$ be a finite group such that $G\subset H$. There is a natural
{\em induction operation} which produces an $H$-variety. Consider the following equivalence relation on
$H\times Z$: $(h_1, x_1)\sim(h_2, x_2)$ ($x_i\in Z$, $h_i\in H$) if and only if there exists $g\in G$ such that
$h_2=h_1g^{-1}$, $x_2=gx_1$. The quotient $\ind_G^H Z:=(H\times Z)/\sim$ carries a natural $H$-action.
The map $x\mapsto (1,x)$ is an embedding of $Z$ into $\ind_G^H Z$ (as a $G$-variety).

\begin{definition}\label{def:Groth-fGr} {\em The Grothendieck ring of quasi-projective varieties with
finite groups actions} is the Abelian group $K_0^{\rm fGr}(\Var)$ generated by the classes $[\calX]$ of
(quasi-projective) varieties with finite groups actions modulo the relations:
 \begin{enumerate}
 \item[(1)] and (2) from Definition~\ref{def:pre-Groth-fGr}.
 \item[(3)] if $(Z,G)$ is a $G$-variety and $G$ is a subgroup of a finite group $H$, then
 $$
 \left[(\ind_G^H Z,H)\right]=\left[(Z,G)\right]\,.
 $$
\end{enumerate}
The multiplication in ${K}_0^{\rm fGr}(\Var)$ is defined by the Cartesian product of varieties
with the natural finite groups action on it.
\end{definition}

One can give another description of the ring $K_0^{\rm fGr}(\Var)$. For an 
embedding $\varphi: G \to H$ (of finite groups) there exist a  group homomorphism
$\ind_G^H$ from the Grothendieck ring $K_0^{G}(\Var)$ of  $G$-varieties to 
the Grothendieck ring $K_0^{H}(\Var)$ of  $H$-varieties: 
$$
[(Z,G)] \mapsto [(\ind_{\varphi(G)}^H Z, H)]
$$
($\ind_G^H$ is not a ring homomorphism).
As an Abelian group $K_0^{\rm fGr}(\Var)$ is the injective limit 
$$
\varinjlim K_0^{G}(\Var)
$$
over the set of the isomorphism classes of finite groups with respect to the inclusions.

The multiplication in $K_0^{\rm fGr}(\Var)$ does not come from the multiplications in $K_0^{G}(\Var)$.
It is defined by
$$
\left[(Z_1,G_1)\right]\cdot\left[(Z_2,G_2)\right]=\left[(Z_1\times Z_2,G_1\times G_2)\right]
$$
with the natural (diagonal) action of $G_1\times G_2$.

This interpretation usually is more convenient for proofs and explanations than the initial one.
Therefore, as a rule, we will use it for such purposes.

\begin{remarks}
 {\bf 1.} 
 There exist two natural ring homomorphisms $i: K_0(\Var)\to K_0^{\rm fGr}(\Var)$ sending $[Z]$ to
 $[(Z, \{e\})]$ and $p: K_0^{\rm fGr}(\Var)\to K_0(\Var)$ sending $[(Z,G)]$ to $[Z/G]$. One has
 $p\circ i= {\rm id}$.
 
 \noindent{\bf 2.}
 In \cite{GP}, E.~Getzler and R.~Pandharipande considered the Grothendieck group of varieties with so-called
$\S$-actions, where $\S=\bigsqcup_{n=0}^\infty S_n$ ($S_n$ is the group of permutations of $n$ elements):
$$
K_0(\Var, \S)=\prod_{n=0}^\infty  K_0^{S_n}(\Var)
$$
with the multiplication $\boxtimes$ induced by 
$$
[(X, S_m)]\boxtimes [(Y,S_n )]=[(\ind_{S_m\times S_n}^{S_{m+n}}(X\times Y), S_{m+n})].
$$
One can see that modulo relation~(3) this multiplication coincides with the Cartesian one.
Therefore there exists a natural homomorphism  
$$
K_0(\Var, \S)\to K_0^{\rm fGr}(\Var).
$$
\end{remarks}

Substituting the word {\em group} by the word {\em semigroup} in Definition~\ref{def:Groth-fGr} one gets the
notion of {\em the Grothendieck semiring $S_0^{\rm fGr}(\Var)$ of quasi-projective varieties with finite groups
actions}. There is a natural (semiring) homomorphism $S_0^{\rm fGr}(\Var)\to K_0^{\rm fGr}(\Var)$.
This notion permits to introduce the notion of effectiveness of a power structure defined over
the ring $K_0^{\rm fGr}(\Var)$.

As in Section~\ref{sec:Orbifold_Euler}, let ${\rm Conj\,}G$ be the set of conjugacy classes of elements of
a group $G$. The conjugacy class of an element $g\in G$ is denoted by $[g]$. If there are several groups
containing $g$, we will indicate the group using the notation $[g]_G$. The centralizer of an element $g\in G$
is denoted by $C_G(g)$.

Let $Z$ be a $G$-variety, let $G$ be a subgroup of a finite group $H$, and let $\ind_G^H Z$ be the induced
$H$-variety. If an element $h\in H$ has a non-empty fixed point set
$\left(\ind_G^H Z\right)^{\langle g\rangle}$
(say, $(h_0, x_0)\in \left(\ind_G^H Z\right)^{\langle g\rangle}$), then there exists $g\in G$ such that
$(hh_0g^{-1}, gx_0)=(h_0,x_0)$, i.e., $h_0^{-1}hh_0=g$, $gx_0=x_0$. This means that $g\in [h]_H$.
Since in the definitions of the orbifold Euler characteristic and of the higher order Euler characteristics
the summation runs over representatives of the conjugacy classes of elements of the group, we can assume
that applying them to the $H$-variety $\ind_G^H Z$ we always take a representative of a conjugacy class
of elements of $H$ belonging to the subgroup $G$.

\begin{lemma}\label{lemma:induction}
 Let $G$, $H$, and $Z$ be as above. Then, for $g\in G$, the spaces with finite groups actions
 $\left(\left(\ind_G^H Z\right)^{\langle g\rangle}, C_H(g)\right)$ and
 $\bigsqcup\limits_{\stackrel{[g']\in {\rm Cong\,}G:}{[g']_H=[g]_H}}
 \left(\ind_{C_G(g')}^{C_H(g')} Z^{\langle g'\rangle}, C_H(g')\right)$ are equivalent.
\end{lemma}

\begin{proof} Let $(h_0,  x_0)in \ind_G^H Z$ be a fixed point of the action of $g:$ this means that 
 $(h_0, x_0)\in \left(\ind_G^H Z\right)^{\langle g\rangle}$. As above, there exist
$g'\in G$ such that $(g h_0 (g')^{-1}, g' x_0)=(h_0, x_0) $, i.e. $h_0^{-1} g h_0 =g',\, g' x_0=x_0$. In particular 
$g'\in [g]_H$. In each conjugacy class      $[g']\in {\rm Cong\,}G:$ such that 
$[g']_H=[g]_H $, let us choose a representative $g'$.

Let $(\ind_G^H Z)_{[g']_G}$ be the set of points of $\left(\ind_G^H Z\right)^{\langle g\rangle}$
represented by pairs $(h,x)$ with $h^{-1}g h \in [g']_G$.
For $h\in H$, let $\{ h\}$ be the class of $h$ in $H/G$, let $Z_{\{h\}}$ be the subvariety of $\ind_G^H Z$
consisting of points of the form $(h,x)$, $x\in Z$. (This subvariety depends only on the class $\{ h\}$ of $h$.) 
Then $(\ind_G^H Z)_{[g']_G}$ is the union of the subvarieties $Z_{\{ h \}}$ with $h^{-1} g h\in [g']_G.$
One has
$\left(\ind_G^H Z\right)^{\langle g\rangle}\subset
\bigsqcup\limits_{[g']\in {\rm Cong\,}G:} (\ind_G^H Z)_{[g']_G}.$
For each chosen $g'$ ($g'\in[g']_G$, $[g']H=[g]_H$), let $h(g')$ be an element of $H$ such that
$(h(g')^{-1} g h(g) =g'$.
(We can choose $g$ itself as a representative of the conjugacy class $[g]$ and $h(g)=e$.)
The intersection of $(\ind_G^H Z)_{[g']_G}$  with $Z_{\{ h(g') \}}$ is, in a natural way, isomorphic to $Z^{\langle g' \rangle}.$
The  centralizer $C_H(g')$ acts on  $ \left(\ind_G^H Z\right)^{\langle g\rangle} \cap(\ind_G^H Z)_{[g']_G},$
the later is the union of the orbits of points from $ \left(\ind_G^H Z\right)^{\langle g\rangle} \cap Z_{\{ h(g') \}}.$
Moreover the subgroup of  $C_H(g')$ preserving $ \left(\ind_G^H Z\right)^{\langle g\rangle} \cap Z_{\{ h(g') \}}$
coincides with $C_G(g')$. Therefore $ \left(\ind_G^H Z\right)^{\langle g\rangle} \cap (\ind_G^H Z)_{[g']_G}$
and $\ind_{C_G(g')}^{C_H(g'} Z^{\langle g' \rangle}$ are isomorphic as
$C_H(g')$-varieties. Since
$$
\left(\ind_G^H Z\right)^{\langle g\rangle}=
\bigsqcup\limits_{{[g']\in {\rm Cong\,}G}\atop{[g']H=[g]_H}} \left(\ind_G^H Z\right)^{\langle g\rangle}\cap
(\ind_G^H Z)_{[g']_G},
$$ 
one has the statement.
\end{proof}

It is easy to see that $\ind_G^H Z/H=Z/G$ and therefore
$\chi(\ind_G^H Z/H)=\chi(Z/G)$. This means that
 \begin{equation}\label{eq:Euler_induction-0}
  \chi^{(0)}(\ind_G^H Z,H)=\chi^{(0)}(Z,G)\,.
 \end{equation}

\begin{theorem}\label{theo:Euler_induction}
 Let $Z$ be a $G$-variety, $G\subset H$. Then for $k\ge 0$ one has
 \begin{equation}\label{eq:Euler_induction}
  \chi^{(k)}(\ind_G^H Z,H)=\chi^{(k)}(Z,G)\,.
 \end{equation}
\end{theorem}

\begin{proof}
 Equation~(\ref{eq:Euler_induction-0}) gives the statement for $k=0$.
 Assume that Equation~(\ref{eq:Euler_induction}) is proven for the values of $k$ smaller than the one
 under consideration. One has
 $$
 \chi^{(k)}(\ind_G^H Z,H)=
 \sum_{[h]\in {\rm Conj\,} H}\chi^{(k-1)}\left((\ind_G^H Z)^{\langle h\rangle},C_H(h)\right)\,.
 $$
 It was shown that the fixed point set $(\ind_G^H Z)^{\langle h\rangle}$ is not empty only if there
 exists $g\in G$ such that $[h]_H=[g]_H$. Lemma~\ref{lemma:induction} implies
 $$
 \chi^{(k)}(\ind_G^H Z,H)=
 \sum_{[g]\in {\rm Conj\,} G}\chi^{(k-1)}\left((\ind_{C_G(g)}^{C_H(g)} Z)^{\langle g\rangle},C_H(g)\right)\,.
 $$
 The induction gives
 $$
 \chi^{(k)}(\ind_G^H Z,H)=
 \sum_{[g]\in {\rm Conj\,} G}\chi^{(k-1)}\left(Z^{\langle g\rangle},C_G(g)\right)=\chi^{(k)}(Z,G). 
 $$
\end{proof}

Together with Proposition~\ref{prop:tamanoi} this gives the following statement.

\begin{corollary}
 The maps $\chi^{(k)}:K_0^{\rm fGr}(\Var) \to \Z$ are ring homomorphisms.
\end{corollary}

\section{$\lambda$-structures on $K_0^{\rm fGr}(\Var)$}\label{sec:lambda_K^fGr}
Let $(Z,G)$ be a complex quasi-projective variety with a pure action (of a finite group $G$).
The Cartesian power $Z^n$ of the variety $Z$ is endowed with the natural actions of the group $G^n$  
(acting component-wise) and of the group $S_n$ (acting by permutations of the components)
and therefore with the action of their  semidirect product $G^n\rtimes S_n =G_n$: the wreath product.  

\begin{definition}\label{Kapranov-groups} The Kapranov zeta function of $(Z,G)$ is 
$$
 \zeta_{(Z,G)}(t):= 1+\sum_{n=1}^\infty [(Z^n, G_n)]\,  t^n \in 1+t K_0^{\rm fGr}(\Var)[[t]].
$$
\end{definition}

The fact that the Kapranov zeta function is well-defined on $K_0^{\rm fGr}(\Var)$
follows form the following statement.

\begin{proposition}\label{prop:kapranov-ind}
Let $(Z,G)$ be a $G$-variety and let $G\subset H$. Then
$$
\zeta_{(Z,G)}(t)=\zeta_{(\ind_{G}^H Z,H)}(t).
$$
\end{proposition}

\begin{proof}
 The coefficient of $t^n$ in $\zeta_{(\ind_{G}^H Z,H)}(t)$ is represented by
$(\ind_{G}^H Z)^n$ with the corresponding $H_n$-action. One obviously has  
 $(\ind_{G}^H Z)^n=\ind_{G^n}^{H^n} Z^n$ with the corresponding action of $H_n$
and therefore 
$(\ind_{G}^H Z)^n=\ind_{G_n}^{H_n} Z^n$
with the corresponding action of $H_n$. 
The relation $(3)$ in the Definition of $K_0^{\rm fGr}(\Var)$ gives
$[(\ind_{G_n}^{H_n} Z^n, H_n)]=[(Z^n, G_n)]$
what is the coefficient of $t^n$ in  $\zeta_{(Z,G)}(t).$
\end{proof}

The Kapranov zeta function possesses the following multiplicativity property.

\begin{proposition}\label{prop:kapranov-product}
Let $(Z_1, G)$ and $(Z_2, G)$ be quasi-projective varieties with actions of a finite group $G$.  
Then one has
$$
 \zeta_{(Z_1 \sqcup  Z_2,G)}(t) = \zeta_{(Z_1,G)}(t)\cdot \zeta_{(Z_2,G)}(t).
$$
\end{proposition}
\begin{proof}
The coefficient of $t^n$  in  $\zeta_{(Z_1 \cup  Z_2,G)}(t)$
is represented by variety  $(Z_1 \cup  Z_2)^n$ with the corresponding $G_n$ action. One can see that
$$
\left((Z_1 \cup  Z_2)^n, G_n \right)= \bigsqcup_{k=0}^n \left( \ind_{G_k\times G_{n-k} }^{G_n} (Z_1^k\times Z_2^{n-k}), G_n \right).
$$

The relation (3) in Definition~\ref{def:Groth-fGr} means that 
\begin{eqnarray}
[\left((Z_1 \cup  Z_2)^n, G_n \right)]&=&\sum_{k=0}^n [\left(Z_1^k\times Z_2^{n-k}, G_k\times G_{n-k}\right)]
\nonumber\\
&=&\sum_{k=0}^n  [(Z_1^k, G_k)][(Z_2^{n-k}, G_{n-k})]\,.\label{eq:kapranov-prod}
\end{eqnarray}
The right hand side of Equation~(\ref{eq:kapranov-prod}) is just the coefficient of $t^n$ in
$\zeta_{(Z_1,G)}(t)\cdot \zeta_{(Z_2,G)}(t)$.
\end{proof}

Propositions~\ref{prop:kapranov-ind} and~\ref{prop:kapranov-product}
imply the following statement.

\begin{corollary}
The Kapranov zeta function $\zeta_{\bullet}(t)$ defines a $\lambda$-structure on the ring
$K_0^{\rm fGr}(\Var)$.
\end{corollary}

\begin{remark}
 In the terms of Definition~\ref{def:fin-groups} the Kapranov zeta function of a variety
 $\calX=\bigsqcup_{i=1}^s (X_i, G_i)$ with a finite groups action is  
 $$
 \zeta_{\calX}(t)= \prod_{i=1}^s  \zeta_{(X_i,G_i)}(t) \in 1+t K_0^{\rm fGr}(\Var)[[t]].
 $$
\end{remark}

For a $G$-variety $Z$, let $\Delta_G\subset Z^n$ (the {\em big $G$-diagonal}) be the set of $n$-tuples
$(x_1, \ldots, x_n)\in Z^n$ with at least of two of $x_i$ from the same $G$-orbit. One has a natural
action of the wreath product $G_n$ on $Z^n\setminus\Delta_G$ (inherited from the action on $Z^n$).

\begin{definition}\label{lambba} 
Let the series $\lambda_{(Z,G)}(t)\in 1+t K_0^{\rm fGr}(\Var)[[t]]$
 be defined by
$$
\lambda_{(Z,G)}(t):=1+\sum_{n=1}^\infty [(Z^n\setminus \Delta_G, G_n)]\,  t^n \in 1+t K_0^{\rm fGr}(\Var)[[t]].
$$
\end{definition}
Just as above the facts that the series $\lambda_{\bullet}(t)$ is well-defined and defines a
$\lambda$-structure on the ring $K_0^{\rm fGr}(\Var)$ follows from the following statements.

\begin{proposition}\label{prop:lambda-ind}
 Let $(Z_1, G)$ and $(Z_2, G)$ be quasi-projective varieties with actions of a finite group $G$.  
Then one has
$$
 \lambda_{(Z_1 \cup  Z_2,G)}(t) = \lambda_{(Z_1,G)}(t)\cdot \lambda_{(Z_2,G)}(t).
$$
Let $(Z,G)$ be a $G$-variety and let $G\subset H$. Then
$$
\lambda_{(Z,G)}(t)=\lambda_{(\ind_{G}^H Z,H)}(t).
$$
\end{proposition}
 
\begin{proof}  The coefficient of $t^n$  in  $\lambda_{(Z_1 \cup  Z_2,G)}(t)$
is represented by variety  $(Z_1 \cup  Z_2)^n\setminus \Delta_G$ with the corresponding $G_n$ action. One has
$$
\left((Z_1 \cup  Z_2)^n\setminus \Delta_G, G_n \right)= \bigsqcup_{k=0}^n
\left( \ind_{G_k\times G_{n-k} }^{G_n} (Z_1^k\setminus \Delta_G) \times (Z_2^{n-k}\setminus \Delta_G ), G_n \right).
$$
 The relation $(3)$ in the Definition~\ref{def:Groth-fGr} of $K_0^{\rm fGr}(\Var)$ gives
\begin{eqnarray*}
 [\left((Z_1 \cup  Z_2)^n\setminus \Delta_G, G_n \right)]&=& \sum_{k=0}^n
[(Z_1^k\setminus \Delta_G) \times (Z_2^{n-k}\setminus \Delta_G ), G_k\times G_{n-k})]\\
\ =&& \sum_{k=0}^n [(Z_1^k \setminus \Delta_G, G_k)] [(Z_2^{n-k} \setminus \Delta_G, G_{n-k})].
\end{eqnarray*}
The right hand side is the coefficient of $t^n$ in $\lambda_{(Z_1,G)}(t)\cdot \lambda_{(Z_2,G)}(t)$.

The arguments for the second part of the Proposition are literally the same as for the Kapranov zeta function
in Proposition~\ref{prop:kapranov-ind}.
\end{proof}

\section{Power structures over $K_0^{\rm fGr}(\Var)$ }\label{sec:power_K^fGr}
The $\lambda$-structures on $K_0^{\rm fGr}(\Var)$ introduced above define  power structures
over the ring.  In all cases up to now (say in $K_0(\Var)$) the power structures defined by the analogues 
of the series $\zeta_{\bullet}(t)$ and $\lambda_{\bullet}(t)$ are the same.
This is not the case here.

In terms of the corresponding power structures (which we denote by $(A(t))^m_\zeta$
and $(A(t))^m_\lambda$ respectively)
one has
$$
 \zeta_{(Z,G)}(t)= ( \zeta_{ (\spec(\, \C\, ), (e))} (t) )_\zeta^{[(Z,G)]}=
 (1+\sum_{n=1}^\infty [(\spec(\, \C\, ), S_n)]\,  t^n)_\zeta^{[(Z,G)]}\,,
$$
$$
 \lambda_{(Z,G)}(t)= ( \lambda_{ (\spec(\, \C\, ), (e))} (t) )_\lambda^{[(Z,G)]}=
 (1+[ (\spec(\, \C\, ), S_1)]\,  t )_\lambda^{[(Z,G)]}\,.
$$

Let us compute the first terms of the series $(1+[(\spec(\, \C\, ), S_1)]\,  t)_\zeta^{[(Z,G)]})$
(to see that it is different form $(1+[(\spec(\, \C\, ), S_1)]\,t)_\lambda^{[(Z,G)]})$. We have 
\begin{eqnarray*}
(1+[(\spec(\,\C\,), S_1)]\,t)&=&(1+[(\spec(\,\C\,), S_1)]\,t+[(\spec(\,\C\,), S_2)]\,t^2+\ldots)\times\\
\times(1-[(\spec(\,\C\,), S_2)]\,t^2+\ldots)&=&\zeta_{(\spec(\,\C\,), S_1)}(t)\cdot
\zeta_{-[(\spec(\,\C\,), S_2)]} (t^2)\cdot\ldots 
\end{eqnarray*}
where the dots mean terms do not influencing the part of degree $\leq 2$.
\begin{eqnarray*}
(1&+&[(\spec(\, \C\, ), S_1)]\,  t)_\zeta^{[(Z,G)]})=
(\zeta_{[(Z,G)]}(t)) \cdot (\zeta_{[(Z,G \times S_2)]}(t^2) )^{-1}\cdot\ldots\\
&=&(1+[(Z,G)]t+[(Z^2,G_2)]t^2+\ldots)(1- [(Z, G\times S_2)]t^2+\ldots) \cdot\ldots\\
&=&1+[(Z,G)]t+([(Z^2,G_2)]- [(Z, G\times S_2)])t^2+\ldots
\end{eqnarray*}
where $S_2$ acts on $Z$ trivially.
Thus one has 
\begin{eqnarray*}
(1&+&[(\spec(\,\C\,), S_1)]\,t)_\zeta^{[(Z,G)]}- (1+[(\spec(\,\C\,), S_1)]\,  t)_\lambda^{[(Z,G)]}\\
&=&([(\Delta_G, G_2)]-[(Z, G\times S_2)])t^2 \mod t^3
\end{eqnarray*}
where $\Delta_G$ is the big $G$-diagonal in $Z^2$. The coefficient of $t^2$ is not equal to zero in 
$K_0^{\rm fGr}(\Var)$ even for the trivial action of the group $G$ on $Z$.
This follows from the fact the isotropy groups of points of the two terms have different orders.

Moreover the coefficient of $t^2$ in $(1+[(\spec(\,\C\,), S_1)]\,t)_\zeta^{[(Z,G)]}$ does 
not belong to the image of the Grothendieck semi-ring $S_0^{\rm fGr}(\Var)$ (due to the same arguments).
This means that the power structure defined by the Kapranov zeta function $ \zeta_{\bullet}(t)$
is not effective.  

\begin{theorem}\label{lambba-efective}
 The power structure over $K_0^{\rm fGr}(\Var)$ defined by the series $\lambda_{\bullet}(t)$ is effective. 
\end{theorem}

\begin{proof}
To prove the statement, we will give a formula for $(A(t))^m$, where
$A(t)=1+[(A_1,G_1)]t+[(A_2,G_2)] t^2+\ldots$, $m=[(M,G)]$, that is, for the case when the coefficients of
the series $A(t)$ and the exponent are classes of varieties with pure actions.  
It is given by the equation
  \begin{eqnarray}\label{eq:power-lamdba}
  \hspace{-20pt}&\ & (A(t))^{m}=\nonumber\\
  \hspace{-20pt}&=&1+\sum_{k=1}^{\infty}\left(
  \sum_{\{k_i\}:\sum_i ik_i=k}
  \left[\left(
  \left(
  M^{\sum_i k_i}\setminus \Delta_G
  \right)
  \times\prod_i A_i^{k_i}, G_{\{k_i\}}
  \right)
  \right]
  \right)
  \cdot t^k\,,
 \end{eqnarray}
where the variety $\left(M^{\sum_i k_i}\setminus \Delta_G\right)\times\prod_i A_i^{k_i}$
is endowed with a finite group action in the following way.
It carries the natural action of the product 
$$
G^{\sum_i k_i}\times\prod_i{G_i^{k_i}}
$$
of the finite groups acting on the components of $M$ and $A_i$.
Besides that there is a natural action of the product  $\prod_i S_{k_i}$ of permutation groups, 
where $S_{k_i}$ acts simultaneously on the components of $M^{k_i}$ and of $A_i^{k_i}$ (that is it acts by
permutation on the components of $(M\times  A_i)^{k_i}$).
The variety $\left(M^{\sum_i k_i}\setminus \Delta_G \right)\times\prod_i A_i^{k_i}$ is endowed the
the action of the group $G_{\{k_i\}}$ generated by these two action: a semidirect product of the
groups indicated above.

It is enough to prove that Equation~(\ref{eq:power-lamdba}) defines a power structure over the ring
$K_0^{\rm fGr}(\Var)$ and that  $\lambda_{(M,G)}(t)=(1+t)^{[(M,G)]}$. After that its effectiveness is obvious.
We have to verify the properties (3), (4) and (5) of Definition~\ref{Power}: all other properties obviously
hold. For that let us give a geometric interpretation of the coefficient of $t^k$ in
Equation~(\ref{eq:power-lamdba}).

Let $\Gamma_A:=\coprod_{i=1}^k A_i$ and let $I_A:\Gamma_A\to\Z$ be the tautological function on $\Gamma_A$
which sends the component $A_i$ to $i$.
The coefficient of $t^k$ in Equation~(\ref{eq:power-lamdba}) is the configuration space of pairs 
$(K,\Psi)$, where $K$ is an ordered finite subset of 
$M$ and $\Psi$ is a map from $K$ to $\Gamma_A$ such that $I_A(\Psi(x))\leq I_A(\Psi(y))$
for $x<y$ (that is several first points of $K$ (let us denote their number by $k_1$) are mapped to  $A_1$,
several subsequent ones (number of then being $k_2$) are mapped to $A_2$, \dots) and 
$$
\sum_{x\in K} I_A(\Psi(x))=\sum_{i} i k_i=k.
$$
The group $G^{\sum_i k_i}\times\prod_i{G_i^{k_i}}$ acts on this configuration space: $G^{\sum_i k_i}$ acts
on the source and $\prod_i{G_i^{k_i}}$ acts on the image.
The group $S_{\sum_{i}  k_i}$ acts by simultaneous permutations on points of $K$ sent to $A_i$
and on there images. This gives an action of the group $G_{\{k_i\}}$.

Property~(3). Let $B(t)=1+[(B_1, G'_1)]t+[(B_2,G'_2)] t^2+\ldots$
Let $C_j=\bigsqcup_{i=0}^j (A_i\times B_{j-i})$ be the variety with a finite groups action representing
the coefficient of  $t^j$ in the product $A(t)B(t)$. (Here $A_0=B_0=1=[(\spec(\,\C\,), (e))]$).

The coefficient of $t^k$ in $(A(t)B(t))^m$ is represented by the configuration space $L_k$ ($L$ for ``left'')
of pairs $(K,\Psi)$, where $K$ is an ordered finite subset of $M$ and $\Psi$ is a map from $K$ to
$\Gamma_C=\bigsqcup_{i,j} A_i\times B_j$ such that $\sum_{x\in K}I_C(\Psi(x))=k$.
Such pair is defined by two pairs  $(K',\Psi')$, $\Psi':K'\to\Gamma_A$, and $(K'',\Psi'')$,
$\Psi'':K''\to\Gamma_B$, where $K=K'\cup K''$.
The coefficient of $t^k$ in $(A(t))^m\cdot(B(t))^m$ is represented by the configuration space $R_k$
($R$ for ``right'') of quadruples $((K',\Psi'), (K'',\Psi''))$, where $K'$ and $K''$ are finite ordered
subsets of $M$, $\Psi':K'\to\Gamma_A$, $\Psi'':K''\to\Gamma_B$ and 
$\sum_{x\in K'}I_A(\Psi'(x))+\sum_{x\in K''}I_B(\Psi''(x))=k$.

Modulo orderings of the sets $K$, $K'$ and $K''$ (i.~e. after factorization by the corresponding permutations)
the varieties $L_k$ and $R_k$ are 
equal. However them themselves differ from each other and even the groups acting on them are different.
In order to prove the property, we shall distinguish   
parts of $L_k$ and $R_k$ which can be identified (without factorization by permutations) and such that
$L_k$ and $R_k$ are (disjoint) unions of varieties obtained from these parts by induction operations.
The relation~(3) in Definition~\ref{def:Groth-fGr} will imply that the classes of $L_k$ and $R_k$ in
$K_0^{\rm fGr}(\Var)$ coincide.

We have $L_k=\bigsqcup_{\bar{k}:\sum  i k_i=k} L_{\bar{k}}$.
Let $k_{i_{\ell}} (\ell=1,2,\ldots)$ be integers such that $\sum_{\ell} k_{i_{\ell}}=k_i.$  
Let $L_{\{k_{i_{\ell}}\}}$ be the subvariety of  $L_{\bar{k}}$
consisting of pairs $(K,\Psi)$ such that among $k_i$ points of $K$ mapped  into $C_i$
there are $k_{i_{\ell}}$  points $x$ the first components  $\pi_1\circ\Psi (x)$ of whose image under $\Psi$
belong to $A_{\ell}$ (and thus the second component belongs to $B_{i-\ell}$).

Let $\widehat{L}_{\{k_{i_{\ell}}\}}$ be the subvariety of $L_{\{k_{i_{\ell}}\}}$  
consisting of pairs  $(K,\Psi)\subset  L_{\{k_{i_{\ell}}\}}$ such that the points of $K$
(of fixed multiplicity $i$) are ordered, say, in the following way:  
if $\ell_1<\ell_2,$ then those points $x$ for whom $\pi_1(\Psi(x))\in A_{\ell_1}$
precede those for whom $\pi_1(\Psi(x))\in A_{\ell_2}$, (the order of the points whose images 
under $\pi_1\circ\Psi$ lie in the same $A_{\ell}$ is arbitrary).

Also we have $R_k=\sqcup_{\bar{k}:\sum_i i k_i=k} R_{\bar{k}}.$
For a collection $\{k_{i_{\ell}}\}$, let  $R_{\{k_{i_{\ell}}\}}$ be the corresponding subvariety of
$R_{\bar{k}}$. Let $\widehat{R}_{\{k_{i_{\ell}}\}}$ be the subvariety of $R_{\{k_{i_{\ell}}\}}$ of
the quadruples $\left((K',\Psi'), (K'' ,\Psi'')\right)$ 
such that the points of $K'$ (of $K''$ respectively) of fixed multiplicity $\ell$ are ordered, say,
in the following way:  
if $i_1<i_2,$ then those for whom $\Psi(x)\in C_{i_1}$ precede those for whom $\Psi(x)\in C_{i_2}$,
the order of the points whose images under $\Psi$ lie in the same $C_i$ is arbitrary.

The varieties $\widehat{L}_{\{k_{i_l}\}}$ and $\widehat{R}_{\{k_{i_l}\}}$ carry natural actions of the group
$$
\prod_{i, \ell} S_{k_{i_{\ell}}}
$$
and they are isomorphic as $\prod_{i, \ell} S_{k_{i_{\ell}}}$-varieties.

Moreover, as a $\prod_{i} S_{k_i}$-variety:
$$
{L}_{\{k_{i_{\ell}}\}}=
\ind_{\prod_{i,\ell} S_{k_{i_{\ell}}}}^{\prod_{i} S_{k_i}} \widehat{L}_{\{k_{i_{\ell}}\}}.
$$
Let $m_{\ell}:=\sum_{i} k_{i_{\ell}}, m_{\ell}':=\sum_{i} k_{i_{i-\ell}}$. 
The group $S_{k_{i_{\ell}}}$ is embedded into
$(\prod_{\ell} S_{m_{\ell}}) \times (\prod_{\ell'} S_{m_{\ell'}})$
permuting the corresponding $k_{i_{\ell}}$ elements among those permuted by $S_{m_{\ell}}$
and the corresponding $k_{i_{\ell}}$ elements among those permuted by $S_{m_{\ell'}}$
simultaneously. Then   
as a $(\prod_{\ell} S_{m_{\ell}}) \times (\prod_{\ell'} S_{m_{\ell'}})$-variety 
$$
{R}_{\{k_{i_{\ell}}\}}=
\ind_{\prod_{i, \ell} S_{k_{i_{\ell}}}}^{(\prod_{\ell} S_{m_{\ell}})\times
(\prod_{\ell'} S_{m_{\ell'}})} \widehat{R}_{\{k_{i_{\ell}}\}}.
$$
The actions of the products of the groups $G$, $G_i$ and $G'_i$ are obviously coordinated.

Modulo the relation (3) in Definition~\ref{def:Groth-fGr} this means that
$[{L}_{\{k_{i_{\ell}}\}}]=[{R}_{\{k_{i_{\ell}}\}}]$ and therefore $[L_k]=[R_k]$.

Property~(4). Let $n=(N,G)$. The coefficient of $t^k$ in $(A(t))^{m+n}$ is represented by the configuration
space $L_k$ of pairs $(K,\Psi)$, where $K$ is a finite subset of $M\sqcup N$ and $\Psi:K\to \Gamma_A$ is a
map such that $\sum_{x\in K}I_A(\Psi(x))=k$. For a collection $\{k_i\}$, $i=1,2,\ldots$, $\sum_i ik_i=k$,
(that is for a partition of $k$), let $k_i'$, $i=1,2,\ldots$, be integers such that $0\le k_i'\le k_i$.
Let $L_{\{k_i\}\{k_i'\}}$ be the subvariety of $L_k$ consisting of pairs $(K,\Psi)$ such that among $k_i$
points of $K$ of multiplicity $i$, the number of points belonging to $M$ is equal to $k_i'$ for $i=1,2,\ldots$
(and thus $(k_i-k_i')$ points of $K$ of multiplicity $i$ belong to $N$).
Let $\widehat{L}_{\{k_i\}\{k_i'\}}$ be the subvariety of $L_{\{k_i\}\{k_i'\}}$ consisting of pairs
$(K,\psi)\in L_{\{k_i\}\{k_i'\}}$ such that the points of $K$ of fixed multiplicity $i$ are ordered in the
following way: first the points of $M$, then the points of $N$.

The coefficient of $t^k$ in $(A(t))^m(A(t))^n$ is represented by the configuration space $R_k$ of
quadruples $\left((K', \Psi'),(K'', \Psi'')\right)$, where $K'$ is an ordered finite subset of $M$,
$K''$ is an ordered finite subset of $N$, $\Psi':K'\to \Gamma_A$, $\Psi'':K''\to \Gamma_A$,
$$
\sum_{x\in K'}I_A(\Psi'(x))+\sum_{x\in K''}I_A(\Psi''(x))=k\,.
$$
Let $R_{\{k_i\}\{k_i'\}}$ be the subvariety of $R_k$ consisting of quadruples
$\left((K', \Psi'),(K'', \Psi'')\right)$
with the number of points of multiplicity $i$ in $K'$ equal to $k_i'$ and the number of points of
multiplicity $i$ in $K''$ equal to $(k_i-k_i')$.

The $\left(\prod_i S_{k_i'}\times \prod_i S_{k_i-k_i'}\right)$-varieties $\widehat{L}_{\{k_i\}\{k_i'\}}$
and $R_{\{k_i\}\{k_i'\}}$ are isomorphic,
$$
L_{\{k_i\}\{k_i'\}}=
\ind_{\prod_i S_{k_i'}\times \prod_i S_{k_i-k_i'}}^{\prod_i S_{k_i}}\widehat{L}_{\{k_i\}\{k_i'\}}
$$
as a $\prod_i S_{k_i}$-variety.
Again the actions of the products of the groups $G$, $G'$ and $G_i$ on these varieties
are obviously coordinated. Therefore
$$
[L_{\{k_i\}\{k_i'\}}]=[R_{\{k_i\}\{k_i'\}}]\,.
$$
This implies that $[L_k]=[R_k]$.

Property~(5). The coefficient of $t^k$ in $(A(t))^{mn}$ is represented by the configuration space
$L_k$ of pairs $(K,\Psi)$, where $K$ is a finite subset of $M\times N$, $\Psi:K\to \Gamma_A$ is a map such
that $\sum_{x\in K}I_A(\Psi(x))=k$. For a fixed function $k(\overline{s})$ ($\overline{s}=(s_1, s_2, \ldots)$)
with non-negative integer values and with $\sum_{\overline{s}}\left(k(\overline{s})\sum_i is_i\right)=k$,
let $L_{k(\bullet)}$ be the subvariety of $L_k$ consisting of the pairs $(K,\psi)$ such that the projection
of $K$ to $M$ consists of $\sum_{\overline{s}}k(\overline{s})$ points and $k(\overline{s})$ of them are such
that the preimage (in $K$) of each of them contains $s_1$ points of multiplicity $1$, $s_2$ points of
multiplicity $2$, \dots 

Let $\widehat{L}_{k(\bullet)}$ be the subvariety of $L_{k(\bullet)}$ consisting of pairs
$(K,\psi)\in L_{k(\bullet)}$ such that the points of $K$ of fixed multiplicity $i$ are ordered
in the following way. One takes an arbitrary order of their projections to $M$ and the order of points
with different projections is the one in $M$ (the order of points with the same projection can be arbitrary).

The coefficient of of $t^k$ in $\left((A(t))^n\right)^m$ is represented by the configuration space
$R_k$ of the following patterns: a finite subset of points of $M$ with different multiplicities (with an
order of them) with a finite subset of $N$ (also ordered) associated to each of them and a map of the
latter to $\Gamma_A$. Such a pattern defines a finite subset of $M\times N$.
For a function $k(\overline{s})$ ($\overline{s}=(s_1, s_2, \ldots)$),
let $R_{k(\bullet)}$ be the subvariety of $R_k$ consisting of the patterns with the composition of the
corresponding finite subset of $M\times N$ of the type described above.

Both on $\widehat{L}_{k(\bullet)}$ and on $R_{k(\bullet)}$ one has the natural action of a semidirect
product $S_{k(\bullet)}$ of the group $\prod_{\overline{s}}S_{k(\overline{s})}$ (acting on the projections
to $M$) and of the group $\prod_{\overline{s}}\left(\prod_i S_{s_i}\right)^{k(\overline{s})}$ (acting on
the preimages of points in $M$). The group $S_{k(\bullet)}$ is embedded into
$\prod_i S_{\sum_{\overline{s}}k(s_i)s_i}$. As $S_{k(\bullet)}$-varieties $\widehat{L}_{k(\bullet)}$
and $R_{k(\bullet)}$ are isomorphic. Moreover,
$$
L_{k(\bullet)}=\ind_{S_{k(\bullet)}}^{\prod_i S_{\sum_{\overline{s}}k(s_i)s_i}}\widehat{L}_{k(\bullet)}\,.
$$
Therefore $[L_{k_{(\bullet)}}]=[R_{k_{(\bullet)}}]$ and thus $[L_k]=[R_k]$.

To show that the power structure~(\ref{eq:power-lamdba}) is defined by the series $\lambda_{(Z,G)}(t)$, we have to prove
that (in terms of the power structure)
$$
\lambda_{(Z,G)}(t)=\left(\lambda_{1}(t)\right)^{[(Z,G)]}=(1+t)^{[(Z,G)]}\,.
$$

The only non-empty summand in the coefficient of $t^k$ in (\ref{eq:power-lamdba}) corresponds to $k_1=k$,
$k_i=0$ for $i>1$ and is represented by the variety $Z^k\setminus\Delta_G$ with the action of
the corresponding wreath product. This proves the statement.
\end{proof}

\begin{remark}
 As any $\lambda$-ring, $K_0^{\rm fGr}(\Var)$ carries the $\lambda$-structures opposite to those
 defined by the series $\zeta_{\bullet}(t)$ and $\lambda_{\bullet}(t)$. The fact that the power structure
 over $K_0(\Var)$ defined by the $\lambda$-structures opposite to those defined by the analogues of the
 series $\zeta_{\bullet}(t)$ and $\lambda_{\bullet}(t)$ is not effective (\cite{Stacks}) implies that
 the power structures over $K_0^{\rm fGr}(\Var)$ defined by the $\lambda$-structures opposite to
 $\zeta_{\bullet}(t)$ and $\lambda_{\bullet}(t)$ are not effective as well.
\end{remark}

\section{Grothendieck ring of varieties with equivariant vector bundles}\label{sec:Groth_vector}
An equivariant vector bundle over a complex quasi-projective $G$-variety $Z$ is a ($\C$-)vector bundle
$p:E\to Z$ with an action of the group $G$ on $E$ commuting with the action on $Z$ and preserving the
vector bundle structure. We shall denote it by $(Z,E,G)$.
The notion of an isomorphism of two varieties with equivariant vector bundles
similar to that in Definition~\ref{def:fGr-isomorphic} is clear.

\begin{definition}
 {\em A quasi-projective variety $\calX$ with an action of finite groups and with an equivariant vector bundle
 $E$ over it} is a variety represented as the disjoint union of subvarieties $X_i$, $i=1,\ldots, s$, with
 actions of finite groups $G_i$ on them and with equivariant vector bundles $E_i$ over them (of different
 ranks in general).
\end{definition}

\begin{definition}
 Two varieties $\calX$ and $\calY$ with finite groups actions and with equivariant vector bundles $E$ and $E'$
 over them are {\em equivalent} if there exist partitions
 $(\calX, E)=\bigsqcup\limits_{i=1}^N (X_{(i)}, E_{(i)}, G_{(i)})$ and
 $(\calY, E')=\bigsqcup\limits_{i=1}^N (Y_{(i)}, E'_{(i)}, G'_{(i)})$ of them such that
 $(X_{(i)}, E_{(i)}, G_{(i)})$ is isomorphic to $(Y_{(i)}, E'_{(i)}, G'_{(i)})$ for $i=1, \ldots, N$.
\end{definition}

\begin{definition}
 {\em The Grothendieck ring of varieties with equivariant vector bundles} is the Abelian group
 $K_0^{\rm fGr}(\Vect)$ generated by the classes $[(\calX,E)]$ of complex (quasi-projective) varieties
 with finite groups actions and with equivariant vector bundles over them modulo the relations:
 \begin{enumerate}
 \item[(1)] if varieties $(\calX, E)$ and $(\calY,E')$ with finite groups actions and with equivariant vector
 bundles $E$ and $E'$ over them are equivalent, then
 $[(\calX, E)]=[(\calY,E')]$;
 \item[(2)] if $\calY$ is a Zariski closed subvariety of $\calX$ invariant with respect to the groups action,
 then $[(\calX, E)]=[(\calY, E_{\vert \calY})]+[(\calX\setminus \calY, E_{\vert \calX\setminus \calY})]$;
 \item[(3)] if $(Z,E,G)$ is a $G$-variety with an equivariant vector bundle and $G$ is a subgroup of a
 finite group $H$, then
 $$
 \left[(\ind_G^H Z,\ind_G^H E,H)\right]=\left[(Z,E,G)\right]\,.
 $$
\end{enumerate}
The multiplication in $K_0^{\rm fGr}(\Vect)$ is defined by the Cartesian product of varieties
with the natural finite groups action and with the sum of the corresponding vector bundles over it.
\end{definition}

In other words
\begin{equation}\label{eq:mult_Vect}
\left[(Z_1,E_1,G_1)\right]\cdot\left[(Z_2,E_2,G_2)\right]=
\left[(Z_1\times Z_2,E_1\times E_2,G_1\times G_2)\right]
\end{equation}
(with the natural action of $G_1\times G_2$).

Just as for the Grothendieck ring $K_0^{\rm fGr}(\Var)$ one can give another description of the ring
$K_0^{\rm fGr}(\Vect)$. For a finite group $G$,
let $K_0^{G}(\Vect)$ be the Grothendieck ring of $G$-varieties with equivariant vector bundles over them
defined in an obvious way.  For an embedding $\varphi: G \to H$ (of finite groups) there exist a group
homomorphism $\ind_G^H$ from the ring $K_0^{G}(\Vect)$ to the ring $K_0^{H}(\Var)$: 
$$
[(Z,E,G)] \mapsto [(\ind_{\varphi(G)}^H Z,\ind_{\varphi(G)}^H E H)]\,.
$$
As an Abelian group $K_0^{\rm fGr}(\Vect)$ is the injective limit 
$$
\varinjlim K_0^{H}(\Vect)
$$
over the set of the isomorphism classes of finite groups with respect to the inclusions.
The multiplication is defined by a formula like~(\ref{eq:mult_Vect}).

\begin{remark}
 There are natural ring homomorphisms $i^{\rm v}:K_0^{\rm fGr}(\Var)\to K_0^{\rm fGr}(\Vect)$
 sending the class of a $G$-variety $Z$ to the same variety with the (trivial) vector bundle of rank $0$
 and $p^{\rm v}:K_0^{\rm fGr}(\Vect)\to K_0^{\rm fGr}(\Var)$: forgetting the vector bundle.
 One obviously has $p^{\rm v}\circ i^{\rm v}={\rm id}$.
\end{remark}

Let $(Z,E, G)$ be a $G$-variety with an equivariant vector bundle over it.
Let $x\in Z$ be a point fixed by an element $g\in G$. The element $g$ acts on the fibre $E_x$ of
the vector bundle as an operator of
finite order. Therefore its action can be represented by a diagonal $(d_x\times d_x)$-matrix ($d_x=\dim E_x$)
with the diagonal entries $\exp(2\pi i q_j)$, $j=1, \ldots, d_x$, where $0\le q_j<1$.

\begin{definition} (cf. \cite{Zaslow}, \cite{Ito-Reid})
 The {\em age} (or the {\em fermion shift}) ${\rm age}_x(g)$ of the element $g$ at the point $x$ is
 $\sum\limits_{j=1}^{d_x}q_j\in\Q$.
\end{definition}

Let $\varphi$ be a (rational) number. As above, for an element $g\in G$, let
$Z^{\langle g\rangle}$ be the fixed point set of $g$. For a rational number $q$, let $Z^{\langle g\rangle}_q$
be the subset of $Z^{\langle g\rangle}$ consisting of the points $x$ with ${\rm age}_x(g)=q$.
(The subset $Z^{\langle g\rangle}_q$ is the union of some of the components of $Z^{\langle g\rangle}$.)

\begin{definition}\label{def:orbifold_Vect}
 {\em The generalized orbifold Euler characteristic of $(Z,E,G)$
 of weight $\varphi_1$} is defined by
 \begin{equation}\label{eq:orbifold_Vect}
  [Z,E,G]_{\varphi_1}:=\sum_{[g]\in {\rm Conj\,}G}\sum_{q\in Q}
  [Z_{q}^{\langle g\rangle}/C_{G}(g)]\cdot\LLL^{\varphi_1q}\in K_0(\Var)[\LLL^s]\,.
 \end{equation}
\end{definition}

 We have to show that the generalized orbifold Euler characteristic is well defined by the class
 of $(Z,E,G)$ in $K_0^{\rm fGr}(\Vect)$.
 Essentially we have to
 show that, for a $G$-variety $Z$ with an equivariant vector bundle $E$ and for a finite group $H$ such that
 $G\subset H$, we have
 $$
 [Z,E, G]_{\varphi_1}=[\ind_G^H Z,\ind_G^H E, H]_{\varphi_1}\,.
 $$
 We will show this a little bit later for higher order generalized Euler characteristics introduced below
 as well. (The generalized orbifold Euler characteristic is one of them.)
 
 Let $\overline{\varphi}=(\varphi_1, \varphi_2, \ldots)$ be a fixed sequence of rational numbers.
 
\begin{definition}
 {\em The generalized Euler characteristic} of $(Z,E,G)$ {of order $k$ of weight $\overline{\varphi}$}
 is defined by
 $$
 [Z,E,G]^k_{\overline{\varphi}}=
 \sum_{[g]\in{\rm Conj\,}G}\sum_{q\in\Q}
 [Z_{q}^{\langle g\rangle}, E_{\vert Z_{q}^{\langle g\rangle}},C_{G}(g)]^{k-1}_{\overline{\varphi}}
 \cdot\LLL^{\varphi_kq}\in  K_0(\Var)[\LLL^s]\,,
 $$
 where $[Z,E,G]^1_{\overline{\varphi}}=[Z,E,G]_{\overline{\varphi}}$ is the generalized orbifold
 Euler characteristic.
\end{definition}
 
Alternatively one can start from $k=0$, defining $[Z,E,G]_{\overline{\varphi}}^0$ as
$[Z/G]\in K_0(\Var)[\LLL^s]$. (It does not depend on the equivalent vector bundle $E$.)
This will be used in the proof of the Theorem~\ref{theo:Gen_Euler_def} below.

\begin{remark}
 One can see, that for $k>1$ (that is for all the generalized Euler characteristics except the
 orbifold one), the definition is shorter (and simpler) than that in \cite{MMJ} for non-singular
 $G$-varieties. One can say that the described setting is to some extend more natural for the definition.
\end{remark}

\begin{theorem}\label{theo:Gen_Euler_def}
 For a fixed $\overline{\varphi}$,
 the generalized orbifold Euler characteristic and the generalized Euler characteristics of higher orders
 are well-defined ring homomorphisms $K_0^{\rm fGr}(\Vect)\to K_0(\Var)[\LLL^s]$.
\end{theorem}

\begin{proof}
 In fact we have to prove only that the generalized Euler characteristics of higher orders are well
 defined as functions on $K_0^{\rm fGr}(\Vect)$. After that their additivity and multiplicativity are
 obvious. (The latter is essentially Lemma~1 in \cite{GPh}.) Thus we have to show that, for a $G$-variety
 $Z$ with an equivariant vector bundle $E$ over it and for a finite group $H\supset G$, one has
 \begin{equation}\label{eq:gen_well_def}
  [\ind_G^H Z, \ind_G^H E, H]^k_{\overline{\varphi}}=[Z, E, G]^k_{\overline{\varphi}}\,.
 \end{equation}
 For $k=0$ this simply means that $(\ind_G^H Z)/H=Z/G$. Assume that (\ref{eq:gen_well_def}) is proven for
 all values of $k$ smaller than the one under consideration. Using Lemma~\ref{lemma:induction} and the
 induction we get
 \begin{eqnarray*}
 &&[\ind_G^H Z, \ind_G^H E, H]^k_{\overline{\varphi}}=
 \sum_{[g]\in {\rm Conj\,} H}\sum_{q\in\Q}
 [(\ind_G^H Z)^{\langle g\rangle}, \ind_G^H E, C_H(g)]^{k-1}_{\overline{\varphi}}\cdot\LLL^{\varphi_kq}=\\
 && \sum_{[g]\in {\rm Conj\,} G}\sum_{q\in\Q}
 [(\ind_{C_G(g)}^{C_H(g)} Z)^{\langle g\rangle}_q,\ind_{C_G(g)}^{C_H(g)} E,
 C_H(g)]^{k-1}_{\overline{\varphi}}\cdot\LLL^{\varphi_kq}=\\
 && \sum_{[g]\in {\rm Conj\,} G}\sum_{q\in\Q}[Z^{\langle g\rangle}, E, C_G(g)]^{k-1}_{\overline{\varphi}}
 \cdot\LLL^{\varphi_kq}=[Z,E,G]^{k}_{\overline{\varphi}}.
 \end{eqnarray*}
\end{proof}

Let $(Z,E,G)$ be a $G$-variety with an equivariant vector bundle, and let $Z_{(d)}=\{x\in Z: \dim E_x=d\}$.
(There are finitely many $d$ with non-empty $Z_{(d)}$.)
One has $[(Z,E,G)]=\sum_d[(Z_{(d)},E_{\vert Z_{(d)}},G)]\in K_0^{\rm fGr}(\Vect)$.

\begin{theorem}
 One has
 $$
 1+\sum_{n=1}^{\infty}[Z^n,E^n,G_n]^k_{\overline{\varphi}}\cdot t^n=
 \prod_{d}\left(\prod\limits_{r_1, \ldots,r_k\geq 1}
 \left(1-\LLL^{\Phi_k(\underline{r})d/2}\cdot t^{r_1r_2\cdots r_k}\right)^{r_2r_3^2\cdots r_k^{k-1}}
 \right)^{-[(Z_d,E_{\vert Z_d},G)]^k_{\overline{\varphi}}}\,,
 $$
 where
 $$
 \Phi_k(r_1,\ldots,r_k)=\varphi_1(r_1-1)+\varphi_2r_1(r_2-1)+\ldots+
 \varphi_kr_1r_2\cdots r_{k-1}(r_k-1).
 $$
\end{theorem}

\begin{proof}
 Since $[\cdot]^k_{\overline{\varphi}}$ is a homomorphism from $K_0^{\rm fGr}(\Vect)$ to $K_0(\Var)[\LLL^s]$,
 one has
 $$
 1+\sum_{n=1}^{\infty}[(Z^n,E^n,G_n)]^k_{\overline{\varphi}}\cdot t^n=
 \prod_{d}\left(1+\sum_{n=1}^{\infty}[Z_d^n,E_{\vert Z_d}^n,G_n]^k_{\overline{\varphi}}\cdot t^n\right).
 $$
 For the $G$-variety $Z_d$ with an equivariant vector bundle $E_{\vert Z_d}$ of constant rank $d$
 the arguments of \cite{GPh} give
 $$
 1+\sum_{n=1}^{\infty}[(Z_d^n,E_{\vert Z_d}^n,G_n)]^k_{\overline{\varphi}}\cdot t^n=
 \left(\prod\limits_{r_1, \ldots,r_k\geq 1}
 \left(1-\LLL^{\Phi_k(\underline{r})d/2}\cdot t^{r_1r_2\cdots r_k}\right)^{r_2r_3^2\cdots r_k^{k-1}}
 \right)^{-[(Z_d,E_{\vert Z_d},G)]^k_{\overline{\varphi}}}.
 $$
 This proves the statement.
\end{proof}

\section{$\lambda$-structures on $K_0^{\rm fGr}(\Vect)$ and power structures over it}
\label{sec:Lambda_power_Vect}
Similar to the case of the Grothendieck ring $K_0^{\rm fGr}(\Var)$, there are two natural $\lambda$-structures
on the Grothendieck ring $K_0^{\rm fGr}(\Vect)$. They are defined by analogues of the series
$\zeta_{\bullet}(t)$ and $\lambda_{\bullet}(t)$. Using the definition of $K_0^{\rm fGr}(\Vect)$
as a limit over the set of finite groups,
we can define them for $G$-varieties with equivariant vector bundles. Let $(Z,E,G)$ be a $G$-variety with an
equivariant vector bundle $E$ over it. Define the following two series:
\begin{eqnarray*}
 \zeta_{(Z,E,G)}(t)&:=&1+\sum_{n=1}^{\infty}[(Z^n, E^n, G_n)]\cdot t^n,\\
 \lambda_{(Z,E,G)}(t)&:=&1+\sum_{n=1}^{\infty}
 [(Z^n\setminus\Delta_G, E^n_{\vert Z^n\setminus\Delta_G}, G_n)]\cdot t^n.
\end{eqnarray*}

\begin{proposition}
 The series  $\zeta_{(Z,E,G)}(t)$ and $\lambda_{(Z,E,G)}(t)$ depend only on the classes $[(Z,E,G)]$ in
 $K_0^{\rm fGr}(\Vect)$.
\end{proposition}

\begin{proof}
 This follows from the equations (for $H\supset G$)
\begin{eqnarray*}
 &\ &[(\ind_{G_n}^{H_n}Z^n,\ind_{G_n}^{H_n}E^n,H_n)]=[(Z^n,E^n,G_n)]\,,\\
 &\ &[(\ind_{G_n}^{H_n}Z^n\setminus\Delta_{H_n},\ind_{G_n}^{H_n}E^n_{Z^n\setminus\Delta_{G_n}},H_n)]=
 [(Z^n\setminus\Delta_{G_n},E^n_{Z^n\setminus\Delta_{G_n}},G_n)]\,,
\end{eqnarray*}
 whose proofs are almost the same as in Propositions~\ref{prop:kapranov-ind} and~\ref{prop:lambda-ind}.
\end{proof}

The fact that these series define $\lambda$-structures on the ring $K_0^{\rm fGr}(\Vect)$
follows from the following statement.

\begin{proposition}
 One has
\begin{eqnarray*}
 \zeta_{(Z_1\sqcup Z_2,E_1\sqcup E_2,G)}(t)&=&\zeta_{(Z_1,E_1,G)}(t)\cdot\zeta_{(Z_2,E_2,G)}(t)\,,\\
 \lambda_{(Z_1\sqcup Z_2,E_1\sqcup E_2,G)}(t)&=&\lambda_{(Z_1,E_1,G)}(t)\cdot\lambda_{(Z_2,E_2,G)}(t)\,.
\end{eqnarray*}
\end{proposition}

The proof is almost the same as in Propositions~\ref{prop:kapranov-product} and~\ref{prop:lambda-ind}.

The reductions of the series $\zeta_{\bullet}(t)$ and $\lambda_{\bullet}(t)$ under the natural
homomorphism $p^{\rm v}:K_0^{\rm fGr}(\Vect)\to K_0^{\rm fGr}(\Var)$ coincide with the $\lambda$-structures
on $K_0^{\rm fGr}(\Var)$ discussed in Section~\ref{sec:lambda_K^fGr}. Since these two $\lambda$-structures
on the ring $K_0^{\rm fGr}(\Var)$ lead to different power structures, the same holds for the
$\lambda$-structures defined by $\zeta_{\bullet}(t)$ and $\lambda_{\bullet}(t)$ on the ring
$K_0^{\rm fGr}(\Vect)$. Again as in the case of the Grothendieck ring $K_0^{\rm fGr}(\Var)$
the power structure defined by $\zeta_{\bullet}(t)$ is not effective, and the one defined by
$\lambda_{\bullet}(t)$ is. (To show the effectiveness of the power structure defined by the series 
$\lambda_{\bullet}(t)$, one can use an analogue of Equation~(\ref{eq:power-lamdba}).
Equivariant vector bundles over the summands in the right hand side of
it are defined in the obvious way.) The power structures over $K_0^{\rm fGr}(\Vect)$ defined by
the $\lambda$-structures opposite to $\zeta_{\bullet}(t)$ and $\lambda_{\bullet}(t)$ are not
effective.

\end{document}